\documentclass[11pt]{amsart}
\usepackage{geometry}                
\usepackage{graphicx}
\usepackage{amssymb}
\usepackage{epstopdf}

\newtheorem{theorem}{Theorem}[section]

\newtheorem{lemma}[theorem]{Lemma}

\theoremstyle{definition}

\newcommand{\N}{\mathbb{N}}
\newcommand{\R}{\mathbb{R}}

\newcommand{\Q}{\mathbb{Q}}

\newcommand{\ve}{\varepsilon}

\title{The Uncountability of the Unit Interval}
\author{Christina Knapp and Cesar E. Silva}

\begin{document}
\maketitle

\section{Introduction}

For any particularly interesting theorem one proof is never enough.  Instead, the first proof sets the challenge to find a more elegant method that illuminates subtle features of the math, is simpler to understand, or even avoids using controversial subjects.    
In this paper we consider a subject that has  attracted the attention of many mathematicians: the uncountability of the real numbers in the unit interval.  We present the most exhaustive collection of proofs of this fact that we know.  These range from Cantor's three published proofs, including his famous diagonalization method, to 
more recent proofs that employ measure theory, game theory, algebra, and analysis.  

\section{Cantor's First Proof}

The first  proofs of the uncountability of the unit interval were given by Georg Cantor (1845--1918).  Cantor's research was motivated in part by his work in trying  to prove the continuum hypothesis, a conjecture about the different sizes of infinity in the real line which he first formulated at the end of his 1878  article \cite{Ca1878}, see  \cite[p. 879]{Ew1996}. If $A$  is
an infinite subset of the real numbers whose  cardinality 
is not that of the natural numbers $\N$, 
  the continuum hypothesis states that $A$ has the cardinality of the real numbers $\mathbb R$.  Although Cantor agonized over this open question for many years, he never successfully answered it---it is now known to be impossible to prove or disprove from the standard axioms of set theory \cite[p.118]{Da1990}, \cite{St2002}.  

Cantor defined two sets to be ``equivalent" if ``it is possible to put them, by some law, in such a relation to one other  that to every element of each one of them  corresponds  one and only one element of the other" (we cite here the translation of his 1895 article in \cite[p. 86]{Ca1952}, though this definition already appears at the start of Cantor's 1878 \cite{Ca1878} article, in which Cantor discusses the surprising fact that different dimension Euclidean spaces are equivalent; see \cite{Go2011} for a recent discussion of this proof ).  Today we call this type of relation a bijection.  We understand what Cantor called ``equivalent" to mean the sets are the same size, or have the same {\it cardinality}. 

The most obvious infinite set  is the set of natural numbers.  We call sets that  can be put into a bijection with the natural numbers {\it countably infinite} and understand these to be the smallest infinite sets.  
Sets that are finite (including the empty set) or countably infinite are said to be {\it countable}. If 
a set is neither finite nor countably infinite we call the set {\it uncountably infinite} or simply  {\it uncountable}.  Proving whether or not a set is countably infinite is equivalent to proving whether or not it is possible to index the set by the natural numbers.  

Cantor published three proofs that show the set of real numbers is uncountable.  The first, which we present below, appeared in an 1874 article  proving the existence of transcendental numbers. Liouville,
in 1844,  had already proved  that transcendental numbers exist, but Cantor's proof is different and he obtains, in a sense that can be made precise, that there are ``more'' transcendental numbers than algebraic numbers. 
  We have modified this proof from Cantor's original notation so as to make it more accessible to the reader (for translations of the original paper see \cite[pp. 620-622]{St2008} or \cite[pp. 839-843]{Ew1996}).  

\begin{theorem}[Cantor---1874, \cite{Ca1874}]\label{T:cantorfirst} If $\{\omega_k\}_{k\in\mathbb N}$ is a sequence of distinct real numbers, then for every interval of real numbers $[\alpha, \beta]$ ($\alpha<\beta$), there is at least one $\eta$ in $[\alpha, \beta]$ that  does not occur in the sequence
$\{\omega_k\}_{k\in\mathbb N}$.\end{theorem}

\begin{proof}  
We will show that $\eta$ exists by defining a sequence of nested closed intervals $I_n=[a_n, b_n],$ where  $a_n $ and $b_n$ are elements in the sequence  $ \{\omega_k\}$, and choosing $\eta$  in the intersection of all the $I_n$.

Let $k_1$ be the first natural number such that $\omega_{k_1}$ is in the interval
$(\alpha, \beta)$. (If $k_1$ did not exist, any $\eta$ in $(\alpha,\beta)$ would satisfy the assertion of the 
theorem.) Similarly, let $k_2$ be the smallest natural number such that $k_2>k_1$ and
$\omega_{k_2}$ is in $(\alpha,\beta)$. Then define  
 \[a_1=\min\{\omega_{k_1}, \omega_{k_2}\}\text{ and }b_1=\max\{\omega_{k_1},\omega_{k_2}\}
 \text{ and set } I_1=[a_1,b_1].\]
   We next choose $\omega_{k_3},\omega_{k_4}$  to be the next-indexed elements in $\{\omega_k\}$  that are in  $(a_1, b_1)$ and set $a_2$ the smallest, $b_2$ the largest and let $I_2=[a_2,b_2]$.  We continue this process, so that if $I_{n-1}=[a_{n-1},b_{n-1}]$ has been defined, then $a_n $ and $b_n$ are the two next-indexed elements of $\{\omega_k\}$ such that $a_n<b_n$ and $a_n, b_n$ are in $(a_{n-1}, b_{n-1})$.   
If the sequence of intervals that is generated is finite, suppose $I_{K}$ is the last one. 
Then we can choose $\eta$ to be any element of $I_K$, completing the proof.  Now
we assume the sequence of intervals is infinite.

By  construction, the intervals are nested in the sense that \[I_1\supset I_2\supset\cdots\supset I_n\supset\cdots.\]
  Consider the sequences $\{a_k\}_{k\in\mathbb N}$ and $\{b_k\}_{k\in\mathbb N}$.  We see the first is strictly increasing and bounded above by any $b_n$ and the second is strictly decreasing and bounded below by any $a_n$.  Therefore their  limits exitst  and we may 
   define them as  \[a^*=\lim_{k\to\infty}a_k\text{ and }b^*=\lim_{k\to\infty}b_k.\]
 (A modern reader would justify the existence of these limits  by appealing to the monotone convergence theorem, a consequence of the Bolzano--Weierstrass theorem; Cantor
 states simply---in the translation of \cite{St2008}---that \lq\lq because they are strictly increasing in size without growing infinite, have a definite limit.") Since  $a_n<b_n$ for all $n$, then  $a^*\leq b^*$ and thus
  the interval $[a^*, b^*]$ is nonempty.  Therefore we can choose a number  $\eta$ in $[a^*, b^*]$.

Finally we prove that $\eta$ is not equal to any of the elements in the sequence
 $\{\omega_k\}$.  Suppose it were, say $\eta=\omega_p$, for some $p\geq 1$.  
 We observe  that $\omega_p$ cannot be in the interior of $I_{p}$. Before the endpoints
 of $I_{p}$ were chosen, there were $2(p-1)$ choices of points that were made from the sequence. 
 Therefore, if $\omega_p$ were in the interior of $I_{p}$ it would have to have been chosen already;
 but     $I_{p}$ is chosen so that the interval is disjoint from 
 all the previous choices. This shows that $\omega_p$ is not in the intersection of all the intervals.
 Therefore $\eta$ is a real number not in the sequence $\{\omega_k\}$, completing the proof. 
  \end{proof}

Cantor's proof depends in a fundamental way on the Bolzano--Weierstrass theorem, or its equivalent, the monotone convergence theorem;  according to Moore this theorem already appears in unpublished lecture notes of Weierstrass dating back to  1865, but it  first appeared in print in an article by Cantor published in 1872 \cite[p. 221]{Mo2008}.  Cantor uses   the Bolzano--Weierstrass theorem
to prove that a decreasing sequence of bounded closed intervals has a nonempty intersection.
Basically using the same idea, with the aid of the notions of supremum and infimum of a set, one
can obtain the following extension, which will be useful later:

\begin{theorem}\label{T:closedbounded}
Let $\{C_n\}_{n\in\mathbb N}$,  be a sequence of closed and bounded sets satisfying 
\[C_1\supset C_2\supset\cdots\supset C_n\supset\cdots.\]
Then their intersection $\bigcap_{n=1}^\infty C_n$ is nonempty.
\end{theorem} 

\begin{proof}
Let $a_n=\inf C_n$ and $b_n=\sup C_n$.  Then $a_n$ and $b_n$ are in $C_n$ 
and a similar argument to that in the proof of Theorem~\ref{T:cantorfirst}  shows that
their intersection is nonempty.
\end{proof}

As Cantor observes,  Theorem~\ref{T:cantorfirst} shows that a nontrivial interval  cannot be mapped bijectively into any sequence, thus showing that the set of real numbers is uncountable.  Interestingly,  Cantor's paper starts with a proof that the set of algebraic numbers is countable, and then notes that Theorem~\ref{T:closedbounded} implies that every interval contains infinitely many transcendental numbers. It is important to note, as has  been  argued by Gray \cite{Gr1994}, that Cantor's theorem, in the case when the sequence 
$ \{\omega_k\}$ consists of the set of  algebraic numbers, provides a constructive way of obtaining a transcendental number; in fact Gray uses this proof to give an algorithm for producing a transcendental number  (we will come back to this after the diagonalization proof). 
The last paragraph in the proof of Theorem~\ref{T:cantorfirst}  proves that
when the sequence $ \{\omega_k\}$ is dense in the original interval, $a^*=b^*$, so in the case of the sequence of 
algebraic numbers, for example, $a_n$ converges to a transcendental number $\eta$.

In his letter to Dedekind of December 7, 1873, Cantor gives a more elaborate proof for constructing the number $\eta$ than the one presented in Theorem~\ref{T:cantorfirst}. Two days later, on December 9, Cantor writes to Dedekind that he has found the simpler proof that he published in 1874  and that is essentially reproduced above (see \cite[pp. 845-846]{Ew1996} for translations of the letters). The reader may refer 
to Gray \cite[p. 827]{Gr1994} for a discussion of his unpublished  proof.

We give another proof that uses  the nested intervals theorem and could be considered a simpler version of this proof. This is, in fact, the first proof presented in \cite{Ox1980}, where he also
modifies it to give another proof of the Baire category theorem \cite{Ox1980}.

\begin{proof}  For concreteness we assume that $[\alpha,\beta]$ is $[0,1]$. 
We construct a number $\eta$ in the intersection of a certain nested sequence of intervals
so that $\eta $ is different from all $\omega_k$.
Split the unit interval into the intervals $[0,\frac13], [\frac13,\frac23]$ and $[\frac23,1]$ and let $I_1$ be 
the interval that does not contain $\omega_1$ (the left-most one if there is more than one choice). Next split $I_1$ into three equal-length subintervals and choose the one that does not contain $\omega_2$. Continue in this way to generate
a sequence of nested closed intervals
\[I_1\supset I_2\supset\cdots\supset I_n\supset\cdots.\]
Then there is a point $\eta$ in the intersection of all these intervals (the point is in fact unique as the length of the intervals decreases to $0$). Finally we observe that $\eta$ is different from all the $\omega_n$. It is clearly different from $\omega_1$ as $\eta$ is in $I_2$, an interval which does not contain $\omega_1$.  Similarly we know that $\omega_n$ is not in $I_{n+1}$, so it cannot equal $\eta$. Thus we have constructed a point that is not in the enumeration. \end{proof}

We will see that the end of the first proof of Cantor's theorem is similar to the logic by which Borel arrives at a contradiction  in his proof of the Heine--Borel theorem (Theorem~\ref{specialHeineBorel}), see \cite[p. 225]{Bo1950}, whose first edition appeared in 1898.   Borel uses Theorem~\ref{specialHeineBorel} in his development of the theory of measure on the line.

\begin{theorem}[Heine--Borel---Special Case]\label{specialHeineBorel}  If 
$\{I_{n}\}_{n\geq 1}$ is 
a sequence of open intervals that 
covers a closed bounded interval $I=[a,b]$, then there exists a finite 
subsequence of the intervals 
$\{I_{n_{k}}\}_{k=1}^{K}$ that covers $I$. 
\end{theorem} 

\begin{proof}  We may assume the open intervals are bounded by writing each unbounded interval as a countable union of bounded open intervals. Write $I_{n} = (a_{n},b_{n})$  for $n\geq 1$. We describe an 
algorithm that will produce a finite subcover. As in the proof of Theorem~\ref{T:cantorfirst}, the ordering of 
the intervals according to their index plays an important role.  We start by choosing a sequence of intervals according to their index. Let $n_1$ be the first natural number such that $a$ is in $I_{n_1}$ and rename this interval   $J_1$.
 If $J_1$ covers $I$, or equivalently if $b$ is in $J_1$, we are done; if not pick the first interval containing 
the right endpoint of $J_1$. Continue in this way. Again as before, this process generates a finite or infinite sequence. If finite we are done. We will show that if the sequence is infinite we arrive at a contradiction.

The integer ${n_{1}}$ has already been chosen  such that 
$a\in (a_{n_{1}},b_{n_{1}})$.  There   exists a smallest integer $n_{2}$ such that 
$b_{n_{1}}\in (a_{n_{2}},b_{n_{2}})$.  In this way we  generate  an infinite sequence 
$n_{1}, n_{2}, \dots$, with $b_{n_k} < b_{n_{k+1}}\leq b$ for all $k\geq 1$.  
This is an increasing   sequence bounded by $b$, so we can   
set \[b^{*}= \lim_{k\to\infty} b_{n_{k}}.\]
 We know $b^{*}\leq b$.  So there 
exists an interval $I_{p}$ such that $b^{*}\in I_{p}$.  Write 
$I_{p}=(a_{p},b_{p})$. We note that $I_p$  is not one of the intervals $I_{n_{k}}$ as 
$b_{n_{k}}\leq b^* < 
b_{p}$ for all $k\geq 1$.

 The contradiction is obtained by noting that $I_p$ is a candidate to be chosen 
 infinitely many times.  Since $b^*$ is in $I_p$ there exists infinitely many $k$ such 
 that $b_{n_k}$ is in $I_p$. Then $I_p$ is a candidate when we choose the smallest 
 integer $n_{k+1}$ such that $b_{n_k}$ is in $I_{n_{k+1}}$.  As 
 $I_p$ was not chosen this implies that infinitely many integers $n_k$ were chosen before $p$, 
 a contradiction. \end{proof} 

It is interesting that Cantor's and Borel's proofs are similar both in their construction of the sequence of intervals $I_n$ and $I_{n_k}$, respectively, and in their process of arriving at a contradiction.  Cantor constructs $I_n$ by choosing  endpoints $a_n,  b_n$ to be the next-indexed elements of the sequence $\{\omega_k\}$ that are in the interior of the previous interval.  Borel chooses his interval to be the next-indexed in the sequence that covers the right endpoint of the previous interval.  Thus both $I_n$ and $I_{n_k}$ are defined recursively by making use of the ordering of the original sequence.  Both proofs arrive at a contradiction by showing that if the sequences  were infinite, it would have to follow that there was an element of the sequences $\{\omega_k\}$ and $\{I_n\}$ that had infinitely many elements preceding it.  This contradicts the fact that the sequences are already ordered by their indices, thus whenever we pick out a specific element $\omega_k$ or $I_k$ we know only $k-1$ elements can come before it. 

Theorem~\ref{specialHeineBorel} is also true when the open cover is not necessarily countable.
In fact, let $\{I_x\}$, for $x$ in an arbitrary index set $\mathcal I$, consist of  a
collection of open intervals covering $[a,b]$, i.e., each $I_x$ is an open interval and each element of $[a,b]$ is in $I_x$ for some
$x$ in $\mathcal I$.
We observe that there exits a countable set $\mathcal I_0\subset\mathcal I$ such that 
the countable subcollection $\{I_x\}_{x\mathcal I_0}$ still covers $[a,b]$. 
In fact, each $I_x$ is a union of open intervals of the form 
$J_{k,n} = (q_k-1/n,q_k+1/n)$ for some  rational numbers $q_k$ in $[a,b]$ and positive integers $n$. Since the  collection of all intervals of the form 
$J_{k,n}$ also forms a cover of $[a,b]$, for each $(k,n)$   
choose  an $I_x$ that contains $J_{k,n}$,  and  rename it $I_{k,n}$; this determines a countable subcover $\{I_{k,n}\}$ of the $\{I_x\}$.

 We observe now that   the Bolzano--Weierstrass theorem
can be obtained as a direct consequence of the Heine--Borel theorem, as extended above
to arbitrary open covers.

\begin{theorem}[Bolzano--Weierstrass]\label{T:Bolzano-Weierstrass}  If $\{a_n\}$ is
a bounded sequence of real numbers, then it has a subsequence that converges.
\end{theorem} 

\begin{proof}  Suppose  that the sequence  is contained in the interval $[a,b]$. If there is no convergent subsequence,  
then
 for each $x$ in $[a,b]$, there exists $\varepsilon_x>0$ such that  the interval $J_x= (x-\varepsilon_x, x+\varepsilon_x)$ 
contains at most finitely many terms of the sequence, i.e., there are only finitely many
integers $n$ such that $a_n$ is in $J_x$. 
As 
$\{J_x\}_{x\in [a,b]}$,  is a open cover of   $[a,b]$ by open intervals,  by  Theorem~\ref{specialHeineBorel} (Heine--Borel) and the remark following, there is a finite subcover
$J_{x_1}, \ldots J_{x_k}$ of $[a,b]$. As each one contains only finitely many terms of the sequence  this is a contradiction since the sequence has  infinitly many terms. Therefore there is a convergent subsequence. 
\end{proof}

\section{Cantor's Second Proof}


Cantor's second proof, published in 1884, uses the notion of a {\it perfect} set, which he had defined in 1883 as a set equal to its set of accumulation points \cite[p. 223]{Mo2008}; or equivalently, a set that is closed and includes its set of accumulation points. Recall that a point $x$ is an accumulation point of a set $A$ if
there is a sequence of distinct elements of $A$ that converges to $x$.  This  paper is also the first to give the definition of closed sets 
\cite[p. 223]{Mo2008}.

In his proof, to construct a point that is not in a given sequence, Cantor uses Cauchy sequences, which he called fundamental sequences. The proof can be simplified, however, using Theorem~\ref{T:closedbounded}, which we do here.  For a proof using Cauchy sequences   see \cite{Fr2010}.
Although Cantor states his theorem for subsets of $\R^n$,  we  continue to treat only subsets of the real line. 


\begin{theorem}[Cantor\thinspace--\thinspace 1884, \cite{Ca1884}]\label{T:cantorsecond} A countable nonempty set of real numbers cannot be perfect.  \end{theorem}

\begin{proof}  
Let $P$ be a countable nonempty set, so we can write $P=\{\omega_n\}_{n\in\mathbb N}$. We will show that if we assume that $P$ is perfect, then there is a point $p'$ in $P$ that is different from $\omega_n$ for all $n$ in $\mathbb N$, which is a contradiction, showing  $P$ could not be perfect.


Let $\varepsilon _1>0$  and consider the interval $B_1=(\omega_1-\varepsilon _1,\omega_1+\varepsilon _1)$; we say $B_1$ is centered at $\omega_1$ and of radius $\varepsilon_1$. As $P$ is perfect, $B_1$ contains infinitely many points of $P$. Then we can choose an open interval $B_2$ 
 centered at a point of $P$ that is in $B_1$, and of radius $\varepsilon_2>0$ such
that $B_2$ does not contain $\omega_2$ and its closure $\overline{B_2}$ is included in $B_1$.  Again we know that $B_2$ contains infinitely many points of $P$. Choose an open interval $B_3$ centered at a point of $P$ that is in $B_2$
and with radius such that $B_3$ does not contain $\omega_3$ and its closure is included in $B_2$.

In this way we generate  a sequence of open intervals $\{B_n\}_{n\in\mathbb N}$ such that
$B_n$ contains a point of $P$, \[\overline{B_n}\subset B_{n+1},\] and $B_n$ does not contain $\omega_n$.

Let 
\[C_n=\overline B_n\cap P.\]
Then the sets $C_n$ are closed and bounded, and decreasing: $C_n\supset C_{n+1}$.
Therefore there exists a point $p'$ in their intersection $\bigcap_{n\geq 1}C_n$. The point $p'$ is in 
$P$ and since it is in $B_n$ for all $n\geq 1$, it is different from $\omega_n$. Therefore $p'$ satisfies 
the properties we were seeking, completing the proof.

 \end{proof}

It follows that the unit interval is uncountable because $[0,1]$  is a perfect subset of the real numbers.\\

The proof of Theorem \ref{T:cantorsecond} uses a similar method as that of Theorem \ref{T:cantorfirst}.  We again construct a sequence of nested closed intervals  and examine the intersection of these intervals to find the element of $P$ that is not in the sequence $\{\omega_k\}$.  But where in the first proof we constructed the $I_n$ by choosing endpoints that are elements of the sequence, here we choose the $B_n$ to exclude all $\omega_i \in \{\omega_k\}, i\leq n$.  


Interestingly, the proof of the Baire category theorem follows almost exactly the same strategy as Theorem \ref{T:cantorsecond}.  The Baire category theorem states that ``if $X$ is a complete metric space, then the intersection of any countable collection of dense open sets in $X$ is dense."  Ren\'e Louis Baire (1874--1932) proved this theorem in his doctoral thesis in 1899.    Instead of using  complete metric spaces we prove   the theorem for  closed subsets of $\mathbb R$,  as a subset of the line is complete (for the Euclidean metric) if and only if it is closed.  For a proof of the general theorem see \cite{Si2008}.

\begin{theorem} [Baire category\thinspace---\thinspace Special Case]\label{T:baire} If $F$ is a closed subset of $\mathbb R$,  then the intersection of any countable collection of dense open sets in $F$ is dense in $F$. \end{theorem}

\begin{proof}
Let $\{G_n\}_{n\in\mathbb N}$ be a countable collection of dense open sets in $F$.  Open in $F$ means that each $G_n$ is of the form $G_n'\cap F$ for some open set $G_n'$  in $\R$. To show that the intersection of the 
$G_n$ is dense in $F$ we need to show that if  $B$ be any nonempty open ball in $F$ (i.e., a set 
of the form $B=\{x\in F: |x-x_0|<\varepsilon_0\}$ for some $x_0$ in $F$ and some $ \varepsilon_0>0$), then
there exists a point $\eta$ in $B$ that is in $G_n$ for all $n $ in $\mathbb N$. 

Since  $G_1$ is dense,  there exists a point  $x_1 $ in  $ G_1 \cap B$.  Construct an open ball $B_1$  centered at $x_1$ with radius $\varepsilon_1<1$  such that
\[\overline B_1 \subset B\cap G_1,\]  
 (for example, choose $\ve_1=\ve_0/2$).  Similarly, since $G_2$ is dense, the set $G_2\cap B_1$ 
contains a point $x_2$.  Construct an open ball $B_2$   centered at $x_2$ with radius $\varepsilon_2<\frac{\ve_1}{2}$ and such that the closure of $B_2$ is contained in $B_1\cap G_2$.  Continue on to get the sequence of balls $B_n$ centered at a point  $x_n$ of $G_n\cap B_{n-1}$ with radius $\varepsilon_n<\frac{\ve_{n-1}}{n}$ such that
 \[\overline B_n \subset B_{n-1}\cap G_n.\]  

The standard proof now proceeds to show that the sequence $\{x_n\}$ is a Cauchy sequence,
and as the space is assumed complete, it must converge to a point we denote $\eta$. In our case we will simplify the argument by using the additional structure of $\R$. The sets $\overline B_n$ are compact  
and nested
\[\overline B_n\supset\overline B_{n+1}.\] 
Therefore there is a point $\eta$ in their intersection.
This point must be in $B$, so the intersection $ \bigcap_{n=1}^\infty G_n \cap B$ is nonempty, completing the proof.


\end{proof}

Now suppose the interval $[0,1]$ could be written as $[0,1]=\bigcup_{n\in\mathbb N} \{\omega_n\}$.
Taking complements we obtain that $\bigcap_{n\in\mathbb N} ([0,1]\setminus\{\omega_n\})$ is the empty set.
But the sets $[0,1]\setminus\{\omega_n\}$ are clearly open and dense in $[0,1]$, 
and the Baire category theorem implies that their intersection is dense in $[0,1]$, so certainly not empty.
Thus the unit interval is uncountable.
More generally, this theorem can be equivalently stated as, ``$F$ cannot be represented by  a countable union of nowhere dense sets." (A set is nowhere dense if the interior of its closure is empty.)   

The Baire category theorem has had many important applications in analysis. 
It is used as a tool for proving existence results. Sets that are countable unions of nowhere dense sets
are regarded as \lq\lq small" and are called sets of {\it first category}. The complement of a first category set
is considered ``large" and we say elements of this set are {\it typical}. Proofs  that use the Baire category method not only show the existence of an element but show the typical behavior. For example, while the existence of nowhere differentiable functions was known following the example of Weierstrass, Banach proved in 1931 that the typical continuous function is nowhere differentiable (i.e., the complement of the  set of functions that are nowhere differentiable is a set of first category, in the set of continuous functions on the unit interval with the uniform metric). For a proof of this theorem and other applications of the Baire category method the reader may refer to \cite{Ox1980}.

\section{Cantor's Third Proof}

Cantor's third and most famous proof, published in 1891, introduces his  ``diagonalization" method and is only four pages long. Cantor starts by referring to his 1874 paper where he proved that the real numbers cannot be put in a one-to-one correspondence with the natural numbers, and then proceeds to say that
``it is possible to give a proof of that theorem without considering the irrational numbers" \cite{Ca1891}   (we have followed the translation in \cite[pp. 920-922]{Ew1996}). He goes on to prove that the set of all infinite sequences on two symbols (Cantor uses $m$ and $w$ for the symbols) is not countable. We  now introduce some 
notation. Think of  $\{0,1\}^\N$ standing for the set of all functions
from $\{0,1\}$ to $\N$; each element $a$ of $\{0,1\}^\N$ can be thought of as an infinite sequence consisting 
of $0$s and $1$s.

\begin{theorem}[Cantor\thinspace--\thinspace 1891, \cite{Ca1891}]\label{T:cantorthird} The set  $\{0,1\}^\N$   is uncountable. \end{theorem}

\begin{proof}
Suppose $\{0,1\}^\N$ were countable.  Then there would  exist a surjection \[f:\mathbb N \to \{0,1\}^\N\]giving  an
enumeration or listing of $\{0,1\}^\N$. For each $n$ in $\N$, we have an element of $2^\N$ denoted
$f(n)$. The value of the function $f(n)$ at a natural number $i$ is denoted $f(n)(i)$. We now define a new
element of $2^\N$. Define $b:\N\to\{0,1\}$ such that 

\begin{align*}
b(i)= 
\begin{cases}
1, \ & \text{if}\   f(i)(i)=0;\\
0, \ &\text{if}\    f(i)(i)=1.
\end{cases}
\end{align*}
Clearly $b$ is in $\{0,1\}^\N$, so there must exits $k$ in $\N$ such that $f(k)=b$. We observe that if
$f(k)(k)=0$, then $b(k)=1$, and if $f(k)(k)=1$, then $b(k)=0$. Each case is a contradiction. Therefore there is no $k$ with
$f(k)=b$, contradicting that $f$ is a surjection. It follows that $\{0,1\}^\N$ cannot be countable.
\end{proof}

Cantor does not give the details of how Theorem~\ref{T:cantorthird} implies that $[0,1]$ is uncountable, 
but one can provide such a proof by indentifying points in $[0,1]$ with their binary  expansions 
as infinite sequences of $0$s and $1$s. One needs to take some care as some points can have two such 
representations.


To obtain transcendental numbers using Cantor's diagonalization method one lists all algebraic numbers in binary expansion in the unit interval say, listing both representations for those numbers
that have two such expansions (the dyadic rationals, i.e., those of the form $n/2^k$). Then the diagonalization method gives the binary expansion of a transcendental number. Gray shows that this
gives an algorithm that is more efficient than the algorithm resulting from Cantor's first proof. In addition,
Gray proves that by this method all transcendentals in $(0,1)$ are obtained \cite[Theorem 3]{Gr1994}. 
In this context one can ask for a listing of all the rationals in $[0,1]$ so that the  algorithm yields
a quadratic irrational, for example. For recent related results the reader may consult \cite{Me2007}.

The diagonalization technique has been used  in many important theorems. One application we mention is its use in  proving that
the halting problem (i.e., whether an abstract computer halts on a giving input) is undecidable. 

Cantor's 1891 paper may contain the first appearance of symbolic spaces, such as the space of
all infinite sequences of $0$s and $1$s. Did Cantor  know that the space $\{0,1\}^\N$ is topologically equivalent (homeomorphic) to the middle thirds Cantor set (see section~\ref{s:measure})? There is a natural metric defined on  $\{0,1\}^\N$. If $x$ and $y$ are  in $\{0,1\}^\N$,with $x\neq y$, let $I(x,y)$ be the first index $i\in\N$ such that 
$x(i)$ and $y(i)$ are different and set \[d(x,y)=\frac{1}{2^I(x,y)},\] and write $d(x,x)=0$.
So points get closer the more they agree from the start. With this metric $\{0,1\}^\N$ is a compact metric space that is perfect \cite[p. 124]{Si2008}, and where points are both open and closed, hence the space satisfies the property of being totally disconnected; it is known that these  properties  characterize a Cantor set, i.e., any two spaces satisfying them are topologically equivalent \cite[p.  216]{Wi1970}.

Cantor's diagonalization method can be very helpful in proving the uncountability of many sets.  The  examples above are by no means the only ones.   Raja has published another proof that uses the diagonalization argument but does not require the negation operation \cite{Ra2005}.

\section{Power Sets}

In his 1891 paper, Cantor also states that his diagonalization proof can be extended 
to prove that ``for a given manifold L we can produce a manifold M whose power is greater than that of L'' 
\cite[p. 922]{Ew1996}, and he goes on to prove by the diagonalization method that there is no surjection from the unit interval to the set of all functions on the unit interval with values in $\{0,1\}$, which can
be identified with the set of all subsets of the unit interval. The {\it power set}, denoted $\mathcal{P}(X)$, of a set $X$ is the set of all subsets of $X$.

Another way to arrive at the uncountability of the unit interval is through power sets.  If one examines
the proof of Theorem~\ref{T:cantorthird} using the natural identification of subsets of $\N$ with sequences of $0$s and $1$s that assigns $i$ to the subset if and only if the $i$th element of the sequence is $1$, ones can see that the subset of $\N$ that is represented by the sequence  $\alpha$  is the set of all elements $i$ in $\N$ such that $i$ is not in $f(i)$. This suggests the proof of the following theorem.

\begin{theorem}\label{T:powercard} (Cantor) The cardinality of $\mathcal{P}(X)$ is strictly greater than the cardinality of $X$. \end{theorem}

\begin{proof} Clearly the cardinality of $\mathcal{P}(X)$ is greater than or equal to the cardinality of $X$.  We proceed by contradiction.  Assume that the cardinality of $\mathcal{P}(X)$ is equal to the cardinality of $X$.  Then there must exist a bijection $f: X\to \mathcal{P}(X)$.  Consider the set \[Y=\{x\in X | x\notin f(x)\}.\]
  $Y$ is clearly an element of $\mathcal{P}(X)$ and thus there must exist some $y\in X$ such that $f(y)=Y$.  Is $y\in Y$?  

Suppose not.  Then $y$ is an element of $X$ that does not map to a set containing itself and thus must be in $Y$ by definition.  This brings us to a contradiction.  Suppose $y$ is in $Y$.  Then $y$ is an element of $X$ that does map to a set containing itself and therefore cannot be an element of $Y$ by definition.  We have a contradiction again.  Therefore there is no valid bijection $f$ that maps $X$ to $\mathcal{P}(X)$.
\end{proof}

To use this theorem to obtain the uncountability of $[0,1]$ we first  use the natural identification of subsets of $\N$ with 
infinite sequences of $0$s and $1$s; this establishes a bijection of $\mathcal P(\N)$ with $2^\N$.  Then as before we observe there is a clear identification of  infinite sequences of $0$s and $1$s  with their binary representation in $[0,1]$ (being careful of the fact that the dyadic rationals have two representations).

It is hard to read the proof of Theorem~\ref{T:powercard} without thinking of Russell's Paradox. Suppose $Y$ is a set consisting of all sets that are not elements of themselves, i.e.,
 \[Y=\{X | X\notin X\}.\]
 Then we ask if $Y$ is an element of itself. If $Y\in Y$ then by the definition of $Y$
  it follows that $Y\notin Y$, a contradiction. On the other hand, if $Y\notin Y$ then again by
  the definition of $Y$ we would have that $Y\in Y$, another contradiction. Thus we obtain a contradiction from the definition of this set. Thus $Y$ cannot be a set.  This paradox was published by Russell in 1903 and caused a re-examination of the foundations of set theory. The way it  is now resolved  is by noting that this unrestricted
  way of defining sets (such as the definition of $Y$), called the schema of comprehension, is not valid; $Y$ is a class that is not a set.  There is now a careful  axiomatization for how sets are defined or constructed; one such system is called the Zermelo--Fraenkel axioms.

\section{Game Proofs}

In 2007,  Matthew Baker \cite{Ba2007} published a proof that essentially reconstructs the logic of Theorem \ref{T:cantorfirst} with a game (Baker also discusses other games, such as the Choquet game, which can be used to prove the Baire category theorem).  The game is as follows. A subset $S$ of the unit interval is given.
There are two players, Alice and Bob, who take turns choosing elements of $S$.  Alice goes first and chooses a  number $a_1$ such that $0<a_1<1$.  Then Bob chooses $b_1$ such that $a_1<b_1<1$.  They continue alternating  and on each turn the player chooses a point between the previous two (i.e.,  on Bob's $n^{th}$ turn he selects $b_n$ such that $a_n<b_n<b_{n-1}$; the choices are such that they generate a nested decreasing sequence of closed intervals $[a_n,b_n]$).  From the monotone sequence theorem we see that because $\{a_n\}$ is  increasing and bounded above, it converges to a real number, denoted $\eta$.  Alice wins the game if $\eta$ is in $S$, Bob wins if $\eta$ is not in $S$.

Baker first proves the following theorem. 

\begin{theorem} If $S$ is a countable set, then there is a winning strategy for Bob.
\end{theorem}

\begin{proof}  Let $S$ be a countable set. If $S$ is empty, then $\eta$ is not in $S$ and any strategy is a winning strategy. So assume it is nonempty  and write it as $S=\{s_i\}_{i\in\N}$ (we do not assume the terms are distinct). Alice starts and chooses
$0<a_1<1$.    Bob's strategy is to check  if $s_1$ is in the open interval $(a_1,1)$. If so, he chooses $b_1=s_1$, if not he chooses the midpoint of the interval $(a_1,1)$. Then for the second move, if $s_2$ is in the 
interval $(a_2,b_1)$, Bob chooses $b_2=s_2$, and if not he chooses the midpoint of the interval. He continues in this way to generate a sequence of choices $b_n$ such that, independently of Alice's choices,  
for each $n$ in $\N$, the point $s_n$ is outside the open interval $(a_n,b_n)$. Since the sequences $\{a_n\}$
and $\{b_n\}$ are strictly increasing and decreasing, respectively,  for every $n\in\mathbb N$,  $a_n<\alpha<b_n$. Therefore $\eta$ is not in $S$ and Bob has a winning strategy.
\end{proof}

If we let $S$ be the unit interval, Alice is clearly guaranteed to win.  Thus the unit interval is uncountable.

Mathematically this proof is very similar to Cantor's first proof.  Alice chooses the left endpoint of each interval $I_n$ and Bob chooses the right endpoint.  When Alice wins it means she has found $\eta$ which is not in the sequence $\{\omega_k\}$.  However this proof is much easier to visualize than the proof of Theorem \ref{T:cantorfirst}.  This gave us the idea  to write Cantor's other proofs with a game argument.  Here is the diagonalization proof:

The game: Alice and Bob are given a set of numbers $S \subset [0,1]$ in their decimal expansion, and such that the expansion does not end in infinitely many $0$s; for example, $1$ is represented as $0.999\ldots$.  Alice and Bob construct a number $\eta$ in the following way.  Alice first chooses a digit $a_1 $ in $\{0,1,..., 9\}$.  Then Bob chooses $b_1$ in $\{0,1,...,9\}$.  They repeat the process to find the second digit with $a_2$ and $b_2$ and so on.  The number they have constructed at the end of the process is $0.z_1z_2z_3...$ where $z_n=a_n+b_n\mod{10}$.  Alice wins if the resulting number is in $S$, Bob wins otherwise.  

\begin{theorem} If $S$ is a countable set, then there is a winning strategy for Bob in the diagonalization game.
\end{theorem}

\begin{proof} Assume $S$ is countable and nonempty and write it as $S=\{s_i\}_{i\in\N}$. Suppose the $j$th digit in the decimal expansion of $s_i$ is $s_{i,j}$.     Whatever Alice chooses for $a_n$ Bob picks $b_n$ such that $a_n+b_n=s_{n,n}+1\mod{10}$ if $s_{n,n}+1\neq 0$, and picks  $b_n$ such that $a_n+b_n=s_{n,n}+2\mod{10}$ otherwise.  Thus he guarantees that for each $n\geq 1$, the $n$th digit of  $\eta$ differs from the $n$th digist of $s_n$. Therefore $\eta$ is not in $S$ and Bob wins the game.
\end{proof}

Again, as before, if $S=[0,1]$ then Alice is guaranteed to win.  Therefore the unit interval is uncountable.
We also note that if $S$ is the set of all algebraic numbers in $[0,1]$, then Bob has a strategy to construct
a transcendental number in $[0,1]$.

Games have been studied in mathematics for some time; in particular two player games where the players alternate picking decreasing intervals are of the Banach--Mazur type and have been used to prove
many results. Schmidt in \cite{Sc1966} introduced Banach--Mazur type games where he proved, among other results, a theorem that implies the existence of uncountably many  numbers that are badly approximable by rationals.

\section{Measure Proof}\label{s:measure}


Another interesting approach is to use results from measure theory to prove that the unit interval is uncountable.  Informally we think of the measure of a set as its length.  We define a set $A$ in $\R$ to have {\it measure zero} if for each $\varepsilon >0$ there exists a sequence of open intervals $\{I_j\}_{n\in\N}$ such that 
\[A\subset \bigcup_{j=1}^\infty I_j\text{ and }\sum_{j=1}^{\infty} |I_j|<\varepsilon,\]
where $|I_j|$ denotes the length of interval $I_j$.  Conceptually this tells us that a set has measure zero if we can cover the entire set with a countable union of arbitrarily small intervals.  

We first prove the following:

\begin{lemma} Every countable set has measure zero. \end{lemma}

\begin{proof} Let $A$ be a nonempty countable set with elements $x_j, j \in \mathbb N$.   
Let $\ve>0$. Define an open cover of $A$ by setting 
\[I_j=\left(x_j-\frac{\varepsilon}{2^{j+2}}, x_j+\frac{\varepsilon}{2^{j+2}}\right).\]  Clearly $|I_j|=\frac{\varepsilon}{2^{j+1}}$ and \begin{align*} \sum_{j=1}^\infty |I_j|&=\sum_{j=1}^{\infty}\frac{\varepsilon}{2^{j+1}}<\varepsilon.\end{align*}  Therefore $A$ has measure zero. \end{proof}

We next prove a lemma which implies that the unit interval does note have measure zero. It says that
if $[0,1]$ (or any closed bounded interval) is covered by a countable union of open intervals, then 
the sum of the lengths of the sequence of intervals is an upper bound for the length of $[0,1]$.
It is intuitively clear but for its proof we will need the version of the Heine--Borel theorem we proved in
Theorem~\ref{specialHeineBorel}.
This lemma can also be found in Borel's development of measure \cite{Bo1950}.

\begin{lemma}\label{L:measure} Suppose $\{I_{j}\}_{j\in\N}$ is a countable  collection of open intervals that covers   $I=[0,1]$:
 \[I\subset\bigcup_{j\in\N} I_j.\] Then \[|I|\leq \sum_{j\in\N}|I_{j}|.\] \end{lemma}

\begin{proof} From Theorem~\ref{specialHeineBorel} we know that there exists a finite subcollection $I_{j_k}, k=1,\ldots\ell,$ that still covers $I$.  

There exists an element of this subcollection, which we rename $(a_1,b_1)$, such that  $a_1<0<b_1$. If $b_1\leq 1$,   there exists an element of this subcollection, which we rename
$(a_2,b_2) $ such that $a_2<b_1<b_2$. Finally we can find an element, which we rename $(a_p,b_p)$,
 with $p\leq\ell$, 
such that $a_p<1<b_p$.
Then,
 \begin{align*} \sum_{i=1}^{k}|I_{j_i}|&\geq (b_{1}-a_{1})+(b_{2}-a_{2})+...+(b_{p}-a_{p}) \\&\geq -a_{1}+(b_{1}-a_{2})+(b_{2}-a_{3})+\cdots+(b_{{p-1}}-a_{p})+b_{p}\\
 &\geq b_p-a_p>1,\end{align*} since each term in parentheses is greater than zero, completing the proof.
\end{proof}

We can now show that the unit interval is uncountable; Stillwell \cite{St2002} attributes this proof to Harnack (1885).

\begin{theorem}\label{T:measure1} If $I=[0,1]$, then $I$ does not have measure zero. \end{theorem}

\begin{proof}  Consider $\varepsilon=\frac{1}{2}$.  Let $\{I_j\}$ be any collection of open bounded intervals that covers $[0,1]$.  
By Lemma~\ref{L:measure} we know that $\sum_{j\in\N}|I_{j}|\geq |I|=1>\varepsilon$, thus we have that $I=[0,1]$ does not have measure zero.
\end{proof}

This tells us that countable sets have measure zero while the unit interval does not, thus the unit interval cannot be the same size as a countable set and must be uncountable. Notice that in Theorem \ref{T:measure1} we have again used Theorem \ref{specialHeineBorel}.  The same argument works for any
interval $[a,b], a<b,$ in place of $[0,1]$.

In a certain sense measure theory makes the result more apparent.  Once you are convinced that countable sets have measure zero (i.e., that you could cover up every element of a sequence with an arbitrarily small interval), then it seems obvious that it must be a smaller set than the unit interval which clearly has length 1.  Be careful though, the converse is not true.  There exist uncountable sets with measure zero. 
The most interesting example of such a set was  defined by Cantor in 1883 (see \cite[p. 329]{Da1990}) and is called the (middle thirds) Cantor set. It is defined on the unit interval 
by successively removing the middle third open subintervals of the closed intervals that are left.
It turns out that the set that is left is not only nonempty but uncountable and has measure zero and is perfect,  see e.g. \cite{Si2008}.

\section{Cauchy Sequences}

In 1969, B. R. Wenner \cite{We1969} published yet another proof of the uncountability of the unit interval using Cauchy sequences.  This proof is based on the construction of the real numbers from the rationals using Cauchy sequences of rational numbers, which was in fact the construction that Cantor proposed \cite[p. 37]{Da1990}. Recall that a  sequence $\{a_n\}_{n\in\N}$ is said to be a Cauchy sequence if for all $\varepsilon>0$ there exists
a natural number $N$ such that $|a_n-a_m|<\varepsilon$ for all $n, m\geq N$. It is clear that a convergent sequence satisfies the Cauchy property.  A Cauchy sequence of rational numbers ``wants to converge," but sometimes the number it would converge to is not rational.  By identifying the limit of the sequence, which is a real number, with that particular sequence, we can construct the set of real numbers from the Cauchy sequences of the rationals.  We must be careful because there exist different Cauchy sequences that converge to the same number.  To avoid this we define an identification on the set $C$ of all Cauchy sequences
on $\Q$ by the 
equivalence relation $\thicksim$ given  by 
\[\{a_n\}\thicksim \{b_n\}\text{ if and only if  }a_n-b_n\text{ converges to zero.}\]  Then a real number corresponds to an equivalence class
of Cauchy sequences $\langle a_n\rangle$. Informally we can state that $x\in \mathbb R$ is represented by $\{a_n\}$ if and only if $\{a_n\}$ converges to $x$.  We can then write  $\mathbb R=\{\langle a_n\rangle:\{a_n\}\in C\}.$  

Our notation works as follows: $a_i(n)$ denotes the $n^{th}$ term of the $i^{th}$ sequence.

\begin{theorem} The set $\mathbb R$ is uncountable. \end{theorem}

\begin{proof}  Assume the set of real numbers is countable and therefore we can write $\mathbb R=\{\alpha_k:k\in\mathbb N\}$, where $\alpha_k=\langle a_k(n) \rangle$.  Our strategy is to define a Cauchy sequence from the rational numbers, $b(k)\in \mathbb Q$ such that $\{b\}\notin \langle \alpha_k \rangle$.

We define two sequences inductively.  The first, $\{N_i\}$ for $ i\in \mathbb N$, we take to be a strictly increasing sequence of natural numbers.  We choose $N_1$ such that for $m, n \in \mathbb N, m,n>N_1$ we have \[|a_1(m)-a_1(n)|<\frac{1}{2^{5}}.\]  Assume $N_{k-1}$ to be defined and choose $N_{k}$ such that for all $m,n\geq N_{k}$ we have \[|a_k(m)-a_k(n)|<\frac{1}{2^{3k+2}}.\]

We use $\{N_i\}$ to define our second sequence, $\{b_n\}$.  Choose $b(1) \in \mathbb Q$ such that \[|b(1)-a_{N_1}(1)|\geq \frac{1}{2^4}. \]  Assume $b_{k-1}$ to be defined and choose $b_{k}$ such that the following are satisfied:  \[|b_{k}-b_{k-1}|<\frac{1}{2^{3(k)}} \text{ and } |b_k-a_{N_k}(k)|\geq \frac{1}{2^{3k+1}}.\]  

We now prove that $\{b_k\}$ is Cauchy and is thus an element of $C$.  Fix $\varepsilon>0$ and choose $N$ such that $\frac{1}{2^{3N}}<\varepsilon$.  
Then if $n\geq m\geq N$
 \begin{eqnarray*} |b_m-b_n| &=&|b_m+(b_{m+1}-b_{m+1})+...+(b_{n-1}-b_{n-1})-b_n| \\ &=& |(b_m-b_{m+1})+(b_{m+1}-b_{m+2})+...+(b_{n-1}-b_n)| \\ &\leq & \sum_{i=m}^n|b_i-b_{i+1}| \leq \sum_{i=m}^n \frac{1}{2^{3(i+1)}} \leq \frac{1}{2^{3N}}<\varepsilon \end{eqnarray*} where the second to last line follows from the definition of $\{b\}$ and the last line follows from $m>N$.

We now show that $\{b_n\} \notin \langle a_k(n) \rangle$.  Let $n \geq N_k$.  $\{N_k\}$ is a strictly increasing function of integers, therefore $N_k\geq k$.  It follows that \begin{eqnarray*} |b_n-a_n(k)| &=& |b_n-b_k+b_k-a_{N_k}(k)+a_{N_k}(k)-a_n(k)| \\ &\geq & |b_k-a_{N_k}|-|a_{N_k}(k)-a_n(k)| -|b_n-b_k|\\ &>& \frac{1}{2^{3k+1}}-\frac{1}{2^{3k+2}}-\sum_{i=k+1}^n|b_i-b_{i-1}| \geq \frac{1}{2^{3k+2}}-\sum_{i=k+1}^n \frac{1}{2^{3(i)}} \\ &>& \frac{1}{2^{3k+2}}-\sum_{i=k+1}^\infty \frac{1}{2^{3(i)}} = \frac{1}{2^{3k+2}}-\frac{1}{7}\frac{1}{2^{3k}} \end{eqnarray*} This tells us the sequence $\{a_n(k)-b_n\}$ cannot converge to zero, thus we have an element of $C$ not included in any of the equivalence classes $ \langle a_k(n) \rangle$.  Therefore our initial assumption was false and we cannot index the real numbers by $\mathbb N$.
\end{proof}

This proof essentially shows that however we try to index the set of real numbers with $\mathbb N$ we are left with a postive distance between any two elements of our sequence.  In the interval defined by these two elements we can always find a real number not included in our sequence.  The construction of the reals from Cauchy sequences on the rationals is interesting in its own right, but complicates the proof, making this approach more difficult to follow than that of Cantor's original proof.

\section{Analytic  Proof}

%
%
%

Perhaps the most unusual recent proof is one that Eliahu Levy \cite{Le2009} adapted from a proof by Bourbaki  in \cite{Bo1949}.  This proof is difficult to follow because it introduces many new terms.  We suggest the reader follow it with pencil in hand to keep track of these new terms and how they come into play.

Our goal will be to define a function that relies on the fact that we are assuming the reals can be indexed by the natural numbers, and then show that this function gives rise to a contradiction.  This will then prove that the reals cannot be indexed by the natural numbers and thus must be uncountable. Instead of using the Bolzano--Weierstrass property or the monotone sequence theory, this uses another completeness property of the real numbers,
equivalent to any of these two, that is the least upper bound or supremum property.  The supremum property states that if $S$ is a nonempty set of real numbers that is bounded above, then there exists a number $\alpha$ in $\R$ such that (1) $\alpha \geq x$ for all $x$ in $S$ (i.e., $\alpha$ is an upper bound of $S$), and (2) for each
$\varepsilon>0$ there exists a number $y$ in $S$ such that $y> \alpha-\varepsilon$
 (i.e., $\alpha$ is the least such upper bound). This number is called the supremum of $S$ and written $\sup S$.

\begin{theorem} The set $\mathbb R$ is uncountable.
\end{theorem}

\begin{proof}

Assume that $\mathbb R$ is countable.  Then we can write $\mathbb R$ as $\{x_n\}_{n\in\N}$.  Define a function $f:\mathbb R \to \mathbb R$ such that for all $x\in \mathbb R, f(x)>0$,  and for any finite subset $F$ in $\R$, the sum \[\sum_{x\in F} f(x)\leq 1.\]  An example of such a function, assuming $\R= \{x_n\}_{n\in\N}$, is $f(x_n)=2^{-n}$ for $n$ in $\N$.  

We need some preliminary definitions.  For each nonempty set   { S} in  $\mathbb R$ define 
\[
{ m(S)} = \sup \left\{\sum_{x\in F}f(x): F\text{ is a finite subset of }S\right\}.\]
From the properties of $f$ it follows that $m(S)$ exists in $\R$  and is bounded above by $1$.
  Finally, define the set \[A= \{x\in\R: x< m(-\infty,x)\}.\]
  Since $A$ nonempty ($0$ is in $A$) and is bounded above by $1$, it follows that it has a supremum
  so we set $c=\sup A$. Letting $\varepsilon=f(c)>0$, we know there exists $y$ in $A$ such that
  \[y>c-f(c).\]
So $c<y+f(c)$ and $y<m(-\infty,y)$.  Also, as $y$ is in $A$, $y\leq c$. 
Now, if $F$ is a finite set in $(-\infty,y+f(c))$, it will consist of elements of $(-\infty,y)$ that do not 
include $c$, plus some elements of $[y,y+f(c))$ that may include $c$. Thus
\[m(-\infty,y+f(c))\geq m(-\infty,y)+f(c)>y+f(c).\]  This means that $y+f(c)$ is in $A$, but as $y+f(c)>c$,  this contradicts that  $c$ is the supremum of $A$.  Therefore our initial assumption is false and $\mathbb R$ is uncountable. \end{proof} 

Levy's proof is unlike any other in this paper.  The use of the supremum shows that there are certain properties of functions defined over uncountable sets that do not transfer to functions over countable sets.  An interesting exercise would be to explore some of the other properties of functions defined over the reals to see if they motivate new proofs.

\section{Conclusion}

The number of proofs of the uncountability of the unit interval demonstrates that this important result has intrigued mathematicians ever since Cantor published his original proof.  The variety of methods people have used illustrates the usefulness of this problem as an introduction to different areas of math.  For instance, for the student new to set theory the game proof or Cantor's diagonalization proof are probably the best approaches to learning about uncountability, but once it is a familiar result it can be the guide to the basics of measure theory,  and analysis.   In this way the uncountability of the reals is important beyond its implications for different sizes of infinity: it links a familiar concept with possibly new ones.  We challenge the reader to extend this result to more fields by finding new proofs that not only enhance our understanding of the structure of the unit interval but also serve as a guide to those areas of mathematics.

\section{Acknowledgements}

This paper started as a project in a independent study course by the first named author supervised by
the second author.
The first named author was supported in part by funds of the Bronfman Science Center at Williams College.
We are indebted to Joe Auslander and Frank Morgan for several comments and in particular to Edward Burger and Steven J.  Miller for a careful reading of the manuscript and several suggestions that improved our paper, and for bringing \cite{Sc1966} to our attention.
We benefited from the translations of Cantor \cite{Ca1874}, \cite{Ca1891} in \cite{Ew1996} and
\cite{St2008}. We would like to thank Bernhard Klingenberg for his help reading Theorem A in
\cite{Ca1884}.

References \cite{Ca1874}, \cite{Ca1884}, \cite{Ca1891} are reprinted in \cite{Ca1932}. Reference 
\cite{Ca1883} is a translation of \cite{Ca1874} and \cite{Ca1883b} is a translation of \cite{Ca1878}. 

\bibliographystyle{plain}
\bibliography{RealRefs}

\end{document}